\documentclass[12pt]{article}
\usepackage{times}
\usepackage{amsmath,amsfonts,amstext,amssymb,amsbsy,amsopn,amsthm,eucal}
\usepackage{txfonts}
\usepackage{dsfont}
\usepackage{graphicx}   

\numberwithin{equation}{section}
  \theoremstyle{plain}
 \newtheorem{theorem}[equation]{Theorem}

 \newtheorem{lemma}[equation]{Lemma}
 \newtheorem{corollary}[equation]{Corollary}

 \theoremstyle{remark}
 \newtheorem{remark}[equation]{Remark}
 \theoremstyle{remark}
 \newtheorem{question}[equation]{Question}
 
\theoremstyle{definition}
 \newtheorem{definition}[equation]{Definition}

\newcommand{\abs}[1]{\left\lvert#1\right\rvert}

\newcommand{\Vol}{{\rm Vol}}

\newcommand{\R}{\mathbb{R}}
\newcommand{\dR}{\mathbb{R}}

\newcommand{\cB}{\mathcal{B}}

\newcommand{\cS}{\mathcal{S}}

\newcommand{\cW}{\mathcal{W}}

\newcommand{\D}{\nabla}
\newcommand{\wto}{\rightharpoonup}

\newcommand{\loc}{\textrm{loc}}
\newcommand{\Lap}{\Delta}
\newcommand{\eps}{\varepsilon}

\begin{document}

\title{Quantitative Stratification and the Regularity of Harmonic Map Flow}

\author{Jeff Cheeger, Robert Haslhofer and Aaron Naber\thanks{J.C. was partially supported by NSF Grant DMS 1005552 and by a Simons Fellowship.}}
\date{}
\maketitle

\begin{abstract}x
In this paper, we prove estimates and quantitative regularity results for the harmonic map flow.
First, we consider $H^1_{\loc}$-maps $u$ defined on a parabolic ball $P\subset M^m\times \R$ and with target manifold $N$, that have bounded Dirichlet-energy and Struwe-energy.  We define a {\it quantitative} stratification, which groups together points in the domain into quantitative weakly singular strata $\cS^j_{\eta,r}(u)$ according to the number of approximate symmetries of $u$ at certain scales,
and prove that their tubular neighborhoods have small volume, namely $\Vol\left(T_r(\cS^j_{\eta,r}(u))\right)\leq Cr^{m+2-j-\eps}$. 
In particular, this generalizes the known Hausdorff estimate $\dim \cS^j(u)\leq j$ for the weakly singular strata of suitable weak solutions of the harmonic map flow.
As an application, specializing to Chen-Struwe solutions with target manifolds that do not admit certain harmonic and quasi-harmonic spheres, we obtain refined Minkowski estimates for the singular set.
This generalizes a result of Lin-Wang \cite{LW1}. We also obtain $L^p$-estimates for the reciprocal of the regularity scale.
The results are analogous to our results for mean curvature flow that we recently proved in \cite{CheegerHaslhoferNaber_MCF}.
\end{abstract}


\section{Introduction and main results}

In this paper, we prove estimates and quantitative regularity results for the harmonic map flow. The results are analogous to our results for mean curvature flow that we recently proved in \cite{CheegerHaslhoferNaber_MCF}.

Recalling the standard setting for the harmonic map flow \cite{Eells_Sampson}, let $M^m$ and $N^n$ be closed Riemannian manifolds and $u_0:M\to N$ be a smooth map; we can assume $N\subset\dR^d$, and we write $A$ for the second fundamental form.
Chen-Struwe \cite{ChenStruwe_existence} proved the existence of weak solutions $u:M\times \R_+\to N$ of the harmonic map flow starting at $u_0$,
\begin{equation}\label{eqn_hmf}
 \partial_t u = \Lap u+A(u)(\D u,\D u)\, ,\qquad\qquad u|_{t=0}=u_0.
\end{equation}
Having a general existence theory, a main question is then to study the regularity of these weak solutions.

Adapting Struwe's monotonicity formula \cite{Struwe_monotonicity} to the setting of their existence proof, Chen-Struwe proved that the parabolic Hausdorff dimension of the singular set $\mathcal{S}\subset M\times \R_+$ is at most $m$, c.f. \cite[Thm. 1.5]{ChenStruwe_existence}.
We recall that the parabolic Hausdorff dimension is the Hausdorff dimension with respect to the space-time metric $d((x,t),(y,s))=\max\{d_M(x,y),\sqrt{\abs{t-s}}\}$; in particular, note that $\dim (M\times \R_+)=m+2$.
About ten years later, based on refined blowup-analysis, Lin-Wang proved that if the target manifold $N$ doesn't admit certain harmonic and quasi-harmonic spheres, then the Hausdorff dimension of the singular set must be smaller \cite{LW1}; this in turn was based on the sophisticated blow-up analysis of Lin in the elliptic setting \cite{Lin_elliptic}.

The goal of the present paper is to make these regularity results more quantitative. The main new ingredient that we develop and apply is the quantitative stratification technique \cite{CheegerNaber_Ricci,CheegerNaber_HarmonicMinimal,CheegerHaslhoferNaber_MCF}; this allows us to turn infinitesimal statements from blowup analysis into more quantitative ones.

\subsection{Quantitative stratification for general targets}\label{ss:quant_reg}

Our first main result, Theorem \ref{t:hmf_quant_strat}, holds without any assumptions on the target manifold, and in fact can be formulated in a very general setting. It applies in particular to Chen-Struwe solutions, but it holds for any $H^1_\loc$-map defined on a parabolic ball and with target $N$ that satisfies the local energy bounds (\ref{dirichlet_bound}) and (\ref{struwe_bound}) below.  In principle, we do not even need to assume that the equation (\ref{eqn_hmf}) is satisfied weakly. 
However in practive the estimates (\ref{dirichlet_bound}) and (\ref{struwe_bound}) arise because $u$ is a weak solution.  
This degree of generality is important for two reasons.  To begin with there is more than one notion of a {\it weak} solution to (\ref{eqn_hmf}).  In reasonable situations, these notions should agree, but as of this point, this has not always been proved.  Additionally, one may arrive at the estimates (\ref{dirichlet_bound}) and (\ref{struwe_bound}) in contexts where (\ref{eqn_hmf}) only holds up to a bounded lower order term, allowing the estimates to apply to a much broader class of situations.
 
To describe this setting more precisely, let $R$ be any sufficiently small radius, say less than one-quarter of the injectivity radius of $M$. We then consider $H^1_\loc$-maps $u$ defined on a parabolic ball $P_{4R}=B_{4R}\times(-(4R)^2,(4R)^2)\subset M\times \dR$, and with target $N$. As usual when studying interior regularity, we will derive estimates on a somewhat smaller parabolic ball, say on $P_{R}$.
Our estimates depend on a bound for the \emph{scale invariant Dirichlet-energy},
\begin{equation}\label{dirichlet_bound}
\sup_{X_0\in P_{2R}}\sup_{r\leq 2R}\, \frac{1}{r^m}\int_{P_r(X_0)}\abs{\D u}^2\, dVdt\leq \Lambda_1,
\end{equation}
and a bound for what we call -- alluding to Struwe's monotonicity formula \cite{Struwe_monotonicity} -- the \emph{Struwe-energy},
\begin{equation}\label{struwe_bound}
\sup_{X_0\in P_{2R}}\int_{P^-_{2R}(X_0)} \abs{(x-x_0)\cdot \D u+2(t-t_0)\partial_{t} u}^2 e^{-\tfrac{\abs{x-x_0}^2}{4\abs{t-t_0}}} \abs{t-t_0}^{-(m+2)/2} dVdt\leq \Lambda_2.
\end{equation}
Here, we used the notation $X=(x,t)$ for points in space-time, and $P^-_{r}(X)=B_r(x)\times (t-r^2,t]$ for backwards parabolic balls. The integral in (\ref{struwe_bound}) is computed in geodesic coordinates.
It follows from \cite[Lemma 4.1, 4.2]{ChenStruwe_existence} and lower semicontinuity that 
Chen-Struwe solutions satisfy (\ref{dirichlet_bound}) and (\ref{struwe_bound}) for some constants $\Lambda_1,\Lambda_2<\infty$ depending only on $M,N$ and energy of the initial map $u_0$.
To keep track of the setting we just described, we write $H^1_{\Lambda_1,\Lambda_2}(P_{4R},N)$ for the space of all $H^1_\loc$-maps with target $N$, defined on a sufficiently small parabolic ball $P_{4R}$, and with energy bounded as in (\ref{dirichlet_bound}) and (\ref{struwe_bound}).

\begin{remark}
 In the literature, there is the notion of \emph{suitable} weak solutions, i.e. weak solutions satisfying a parabolic stationarity condition, c.f. \cite{Feldman,ChenLiLin_suitable}.
However, it is not clear whether or not the weak solutions constructed by Chen-Struwe actually satisfy this condition. This is one of the reasons, why we have chosen a more general setting merely requiring (\ref{dirichlet_bound}) and (\ref{struwe_bound}). 
\end{remark}

Let $u\in H^1_{\Lambda_1,\Lambda_2}(P_{4R},N)$ be a $H^1$ function with bounded scale invariant Dirichlet energy and Struwe-energy. In the quantitative stratification we will group the points in $P_{R}$ into different strata according to their number of approximate symmetries at bounded scales. To motivate the definition let us first discuss the nonquantitative notion of weak tangents:
Given $X_0\in P_R$ and $s<R$ we consider the rescaled map $u_{X_0,s}:P_1(0)\subset \dR^m\times \dR\to N$,
$u_{X_0,s}(x,t)=u(\exp_{x_0}(sx),t_0+s^2 t)$.
Note first that by combining (\ref{dirichlet_bound}) and (\ref{struwe_bound}) we obtain scale invariant bounds
\begin{equation}\label{dirichlet2_bound}
\sup_{X_0\in P_{R}}\sup_{r\leq R}\, \frac{1}{r^{m-2}}\int_{P_r(X_0)}\abs{\partial_t u}^2\, dVdt\leq C(\Lambda_1,\Lambda_2).
\end{equation}
By (\ref{dirichlet_bound}) and (\ref{dirichlet2_bound})
for every \emph{blowup sequence} $u_{X_0,s_\alpha}$, $s_\alpha\to 0$ there exists a subsequence such that $u_{X_0,s_\alpha}\wto \varphi$ weakly in $H^1_{\loc}$. Such a limit $\varphi$ is called a \emph{weak tangent}, and by  (\ref{struwe_bound}) every weak tangent is \emph{backwardly selfsimilar}, i.e.
$\varphi(x,t)=\varphi(\lambda x,\lambda^2 t)$ for all $x\in\dR^m,t<0$ and all $\lambda>0$.
It is important to keep track of the number of symmetries of $\varphi$. The number of spatial symmetries is the maximal $d$ for which there exists a $d$-plane $V\subset\dR^m$ such that $\varphi(x,t)=\varphi(x+v,t)$ for all $x\in\dR^m,t<0$ and all $v\in V$. With respect to the behavior in the time direction, we have to distinguish between the following three cases.
\emph{(i) static:} $\varphi$ is independent of $t$ for all $t\in\dR$; \emph{(ii) quasi-static:} $\varphi$ is independent of $t$ up to some time $T\in[0,\infty)$ but not for all $t$; \emph{(iii) shrinking:} $\varphi$ is not independent of time on $(-\infty,0]$.
Note that in the shrinking case we have $\varphi(x,t)=\psi(x/\sqrt{-t})$ for some function $\psi$ that is not radially invariant. We then consider the number of space-time symmetries
\begin{equation}
D(\varphi)= \left\{ \begin{array}{ll}
d(\varphi) & \textrm{if $\varphi$ is shrinking or quasi-static}\\
d(\varphi)+2 & \textrm{if $\varphi$ is static}\\
\end{array} \right.
\end{equation}

We say that $\varphi$ is \emph{$j$-selfsimilar} if it is backwardly selfsimilar and $D(\varphi)\geq j$.

\begin{remark}
Consider a time-dependent map that is equal to some given stationary harmonic map up to some time $T$, and constant for later times. 
This map satisfies  (\ref{dirichlet_bound}) and (\ref{struwe_bound}), and the blowups at the truncation time $T$ are quasi-static. However, it seems to be unknown if 
quasi-static blowups can actually occur in any "interesting" situation, e.g. for Chen-Struwe solutions with smooth initial data; it is known that they cannot occur if the target doesn't admit harmonic two-spheres.\footnote{The situation is similar for the mean curvature flow: There, quasi-static planes occur when one truncates a Brakke flow at some time. Excluding this "trivial" example, it seems unknown if the quasi-static planes can actually occur as blowups of solutions with smooth embedded initial data; it is known that they cannot occur in the mean convex case.}
\end{remark}

One could then proceed by grouping points $X\in P_R$ into weakly singular strata $\mathcal{S}^j(u)$ according to the number of symmetries of the weak tangents. Namely, $X\in \mathcal{S}^j(u)$ if no weak tangent $\varphi$ at $X$ has $D(\varphi)>j$. However, as mentioned above, we instead group points together into quantitative weakly singular strata $\mathcal{S}^j_{\eta,r}(u)$ according to the number of approximate symmetries at certain scales:

\begin{definition}\label{d:hmf_singularstrata}
 For each $\eta>0$ and $0<r<R$, we define the $j$-th quantitative weakly singular stratum
\begin{align}
\cS^j_{\eta,r}(u):=\left\{X\in P_R : \int_{P_1(0)}\abs{u_{X,s}-\varphi}^2>\eta\,\,{\rm for\,\, all}
\,\,s\in[r,R]\,
{\rm and \,\,all}\, (j+1)\text{-selfsimilar } \varphi\right\}.
\end{align}
\end{definition}

\begin{remark}
 In particular, this recovers the standard stratification via $\mathcal{S}^j(u)=\bigcup_\eta\bigcap_r\mathcal{S}^j_{\eta,r}(u)$.
\end{remark}

We write $\Vol$ for the $m+2$ dimensional (parabolic) Hausdorff measure on space-time, and $T_r$ for $r$-tubular neighborhoods with respect to the parabolic metric. Our main quantitative stratification theorem gives estimates for the volume of tubular neighborhoods of the quantitative weakly singular strata:

\begin{theorem}
\label{t:hmf_quant_strat}
If $u:P_{4R}\subset M\times\dR \to N$ is an $H^1_\loc$-map with energy bounded as in (\ref{dirichlet_bound}) and (\ref{struwe_bound}), then the $j$-th quantitative weakly singular stratum $\cS^j_{\eta,r}(u)$ satisfies the volume estimate
\begin{equation}\label{hmf_volest}
\Vol\left(T_r(\cS^j_{\eta,r}(u))\right)\leq Cr^{m+2-j-\varepsilon}\qquad\qquad (0<r<R)\, ,
\end{equation}
for some constant $C<\infty$ depending only on $\varepsilon,\eta,\Lambda_1,\Lambda_2$ and the geometry of $B_{4R}\subset M$ and $N$.
\end{theorem}

\begin{remark}
 In particular, this generalizes the known Hausdorff dimension estimate $\dim \mathcal{S}^j(u)\leq j$ for suitable weak solutions of the harmonic map flow in various ways.
 First, it replaces the weakly singular strata $\mathcal{S}^j(u)$ by the more effective quantitative weakly singular strata $\mathcal{S}^j_{\eta,r}(u)$. Second, it shows that tubular neighborhoods have small volume, i.e. it improves the Hausdorff estimate to a Minkowski estimate. Third, it applies to a much larger class of maps.
\end{remark}

\subsection{Higher regulartiy for certain targets}\label{ss:quant_reg_certain_targets}

Following Lin \cite{Lin_elliptic} and Lin-Wang \cite{LW1}, we consider target manifolds $N$ that do not admit harmonic two-spheres, i.e. no nonconstant smooth harmonic maps $S^2\to N$.
By \cite[Prop. 4.1]{LW1} this will allow us to upgrade weak convergence in $H^1_{\loc}$ to strong convergence (see also the discussion on page 788 of \cite{Lin_elliptic}).

\begin{remark}
 By Sacks-Uhlenbeck \cite[Thm. 5.7]{SacksUhlenbeck} the assumption that the target does not admit harmonic two-spheres implies that its universal cover must be contractible.  See Section \ref{s:examples} for more on this. 
\end{remark}

To state our higher regularity results, we introduce the following definitions.

\begin{definition}[regularity scale]
\label{d:stationary_harmonic_reg_scale}
We define the regularity scale of a harmonic map flow $u:M\times I\to N$ at a point $X=(x,t)$ by
\begin{align}
 r_u(X):=\sup\left\{r\geq 0: \sup_{P_r(X)} r|\nabla u|+r^2|\nabla^2 u|\leq 1\right\}.
\end{align}
\end{definition}

\begin{remark}\label{r:hmf_reg_scale}
The regularity scale controls all other geometric quantities. Indeed, standard interior estimates give scale invariant bounds for all (spatial and time) derivatives on a somewhat smaller parabolic ball, say on $P_{r_{u}(X)/2}(X)$.
\end{remark}

\begin{definition}[$r$-bad set]\label{def_bad_set}
Given $u:M\times I\rightarrow N$ and $r>0$ we define the $r$-bad set
\begin{align}
 \cB_{r}(u):= \{X\in M\times I: r_{u}(X)\leq r\}.
\end{align}
\end{definition}

Our second main theorem shows that the bad set is quite small, provided the target doesn't admit harmonic two-spheres, and even smaller provided the target doesn't admit certain harmonic and quasi-harmonic spheres. We recall that a quasi-harmonic $\ell$-sphere is a nonconstant smooth map $\psi:\R^\ell\to N$ that is a critical point of the functional $\int_{\R^\ell}\abs{D\psi}^2 e^{-x^2/4}$.
Such maps $\psi$ correspond to selfsimilarly shrinking solutions $u(x,t)=\psi(x/\sqrt{-t})$ of the harmonic map flow.

\begin{theorem}
\label{t:harmonic_map_flow_higher_regularity}
Let $M^m$, $N^n$ be closed manifolds and $u:M\times\R_+\to N$ a Chen-Struwe solution of the harmonic map flow (\ref{eqn_hmf}).
If $N$ does not admit harmonic two-spheres, then the bad set (c.f. Definition \ref{def_bad_set}) satisfies the volume estimate
\begin{equation}
\Vol\left(T_r(\cB_{r}(u))\cap \{t<T\}\right)\leq Cr^{4-\varepsilon}\qquad (0<r<1),
\end{equation}
for some constant $C=C(M,N,u_0,T,\varepsilon)<\infty$. In particular, the parabolic Minkowski dimension of the singular set $\mathcal{S}\subset M\times\R_+$ is at most $m-2$.
More generally, if $N$ doesn't admit harmonic $\ell$-spheres for $\ell=2,\ldots, k$ and quasi-harmonic  $\ell$-spheres for $\ell=3,\ldots, k-1$, then
\begin{equation}
\Vol\left(T_r(\cB_{r}(u))\cap \{t<T\}\right)\leq Cr^{k+2-\varepsilon}\qquad (0<r<1),
\end{equation}
and the parabolic Minkowski dimension of the singular set is at most $m-k$.
\end{theorem}

\begin{remark}
 Theorem \ref{t:harmonic_map_flow_higher_regularity} is a quantitative version of \cite[Thm. 4.3]{LW1}. In particular, it improves the Hausdorff estimate to a Minkowski estimate.
\end{remark}

As an immediate consequence of Theorem \ref{t:harmonic_map_flow_higher_regularity}, we obtain $L^p$-estimates for the reciprocal of the regularity scale. In particular, this gives integral estimates for the derivatives of $u$.

\begin{corollary}\label{c:harmonic_map_flow_higher_regularity}
Let $M^m$, $N^n$ be closed manifolds and $u:M\times\R_+\to N$ a Chen-Struwe solution of the harmonic map flow (\ref{eqn_hmf}).
If $N$ doesn't admit harmonic $\ell$-spheres for $\ell=2,\ldots, k$ and quasi-harmonic  $\ell$-spheres for $\ell=3,\ldots, k-1$, then
for every $p<k+2$ we have
$$
\int_0^T\int_M r_{u}^{-p} dV dt\leq C,
$$
for some $C=C(M,N,u_0,T,p)<\infty$. In particular, there are estimates for the derivatives of $u$,
\begin{equation}
\label{er}
\int_0^T\int_M |\nabla^j u|^{\frac{p}{j}} dV dt \leq C_j ,
\end{equation}
for some constants $C_j=C_j(M,N,u_0,T,p)<\infty$ ($j=1,2\ldots$).
\end{corollary}

We end this section with the following particularly important corollary.  In \cite{CheegerNaber_HarmonicMinimal}, the first general higher derivative estimates for {\it minimizing} harmonic maps were proved. By using Theorem \ref{t:harmonic_map_flow_higher_regularity} we can prove the same estimates for {\it stationary} harmonic maps when the target space does not admit any harmonic $2$-spheres.  Specifically:
\begin{corollary}\label{t:stationary_regularity}
Let $M$, $N$ be closed manifolds and $u:M\to N$ a stationary harmonic map with $\int_M |\nabla u|^2 < \Lambda$.  If $N$ does not admit any harmonic $\ell$-spheres for $\ell=2,\ldots, k$, then for every $p<k+2$ we have
$$
\int_M reg_{u}^{-p} dV\leq C,
$$
for some $C=C(M,N,\Lambda,p)<\infty$, where $reg_u(x)\equiv \sup_{r\leq 1}\{\sup_{B_r(x)} r|\nabla u|+r^2|\nabla^2 u|\leq 1\}$ . In particular, there are estimates for the higher derivatives of $u$,
\begin{equation}
\label{er}
\int_M |\nabla^j u|^{\frac{p}{j}} dV \leq C_j ,
\end{equation}
for some constants $C_j=C_j(M,N,\Lambda,p)<\infty$ ($j=1,2\ldots$).
\end{corollary}

\textbf{Organization of the paper.} In Section \ref{s:hmf_decomposition}, we prove Theorem \ref{t:hmf_quant_strat}.
In Section \ref{s:hmf_higher_regularity}, we prove an $\eps$-regularity lemma, and use this in combination with Theorem \ref{t:hmf_quant_strat} to prove Theorem \ref{t:harmonic_map_flow_higher_regularity}.  In Section \ref{s:examples}, we provide some examples of compact manifolds which do not admit harmonic $2$-spheres.  In particular, we provide examples which do not have nonpositive sectional curvature.\\

\textbf{Acknowledgements.} 
We are grateful to  Fanghua Lin and Harold Rosenberg for several helpful conversations.

\section{Volume estimates for the quantitative strata}\label{s:hmf_decomposition}

In this section, we prove Theorem \ref{t:hmf_quant_strat} following closely our argument for Brakke flows in \cite{CheegerHaslhoferNaber_MCF}. We first prove a quantitative rigidity lemma and decompose $P_R$ into a union of sets, according to the behavior of points at different scales. By virtue of a quantitative differentiation argument, we show that the number of sets in this decomposition grows at most polynomially. We then establish a cone-splitting lemma for weak tangents and prove, roughly speaking, that at their good scales points in  $\cS^j_{\eta,r}(u)$ line up along at most $j$-dimensional subspaces. Finally, we conclude the argument by constructing a suitable covering of $\cS^j_{\eta,r}(u)$ and computing its volume.
For ease of notation, in the following we will pretend $R=1$ and $P_4=P_4(0)\subset \dR^m\times \dR$; the general case works similarly.

\subsection{Energy decomposition}\label{ss:energy_decomposition}

The goal of this subsection is to decompose $P_1$ into a union of sets $E_{T^\beta}$, according to the behavior of points at different scales. As in \cite{CheegerNaber_Ricci,CheegerNaber_HarmonicMinimal,CheegerHaslhoferNaber_MCF} it will be of crucial importance that we can deal separately with each individual set $E_{T^\beta}$, all of whose points have the same $\{0,1\}$-valued $\beta$-tuple $T^\beta$ of good and bad scales.

\begin{definition}
\label{d:hmf_almost_selfsimilar}
A map $u\in H^1_{\Lambda_1,\Lambda_2}(P_{4r}(X),N)$ is \emph{$(\varepsilon,r,j)$-selfsimilar} at $X=(x,t)$ if there exists a $j$-selfsimilar map $\varphi$ such that
$$
\int_{P_1(0)}\abs{u_{X,r}-\varphi}^2<\varepsilon
$$
If $\varphi$ is shrinking with respect to a plane $V^j$, we put $W_X=(x+V)\times\{t\}$. If $\varphi$ is quasi-static with respect to $V^j$ up to time $T$, we put $W_X=(x+V)\times (-\infty,t+r^2T]$. If $\varphi$ is static with respect to $V^{j-2}$, we put $W_X=(x+V)\times\dR$.  We say that $u$ is {\it $(\varepsilon,r,j)$-selfsimilar at $X$ with respect to $W_X$}.
\end{definition}

\begin{definition}
\label{d:hmf_rel_struwe_energy}
Given $u\in H^1_{\Lambda_1,\Lambda_2}(P_{4},N)$, for $X_0\in P_1$ and $1/2 > r_1>r_2$, we define the \emph{$(r_1,r_2)$-Struwe energy} by
\begin{equation}
\cW_{r_1,r_2}(u,X_0):=\int_{P^-_{r_1}(X_0)\setminus P^-_{r_2}(X_0)} \abs{(x-x_0)\cdot \D u+2(t-t_0)\partial_{t} u}^2 e^{-\tfrac{\abs{x-x_0}^2}{4\abs{t-t_0}}} \abs{t-t_0}^{-(m+2)/2} dVdt.
\end{equation}
\end{definition}

\begin{lemma}[Quantitative Rigidity]
\label{t:hmf_monotone_rigidity}
For all $\varepsilon >0$, $\Lambda_1,\Lambda_2<\infty$, $m$ and $N$ there exists $\delta=\delta(\varepsilon,\Lambda_1,\Lambda_2,m,N)>0$, such that if $u\in H^1_{\Lambda_1,\Lambda_2}(P_{4},N)$
satisfies
\begin{align}\label{e:current_almostrigid}
\cW_{r,\delta r}(u,X)\leq \delta,
\end{align}
for some $X\in P_1$ and some $r\in(0,1/2)$, then $u$ is $(\varepsilon,r,0)$-selfsimilar at $X$.
\end{lemma}

\begin{proof}
If not, then there exist a sequence of maps $u_\alpha\in H^1(P_{4},N)$ with
\begin{equation}\label{contr_assumpt}
\int_{P^-_{1}\setminus P^-_{1/\alpha}} \abs{x\cdot \D u_\alpha+2t\partial_{t} u_\alpha}^2 e^{-\tfrac{\abs{x}^2}{4\abs{t}}} \abs{t}^{-(m+2)/2} dVdt\leq 1/\alpha,
\end{equation}
and scale invariant bounds as in (\ref{dirichlet_bound}) and (\ref{dirichlet2_bound}), but such that 
\begin{equation}\label{hmf_biggereps}
\int_{P_1}\abs{u_{\alpha}-\varphi}^2\geq\varepsilon
\end{equation}
for all $0$-selfsimilar $\varphi$. However, it follows from (\ref{dirichlet_bound}), (\ref{dirichlet2_bound}) and (\ref{contr_assumpt}) that, after passing to a subsequence, $u_\alpha\wto\varphi$ weakly in $H^1$ for some $0$-selfsimilar $\varphi$. Since the convergence is strong in $L^2$, for $\alpha$ large enough we obtain a contradiction with (\ref{hmf_biggereps}).
\end{proof}

Let $u\in H^1_{\Lambda_1,\Lambda_2}(P_4,N)$ and $X\in P_1$. Given constants $0<\gamma<1/2$ and $\delta>0$ and an integer $q<\infty$ (these parameters will be fixed suitably in Section \ref{hmf_conclusion}), let $K$ be the number of $\alpha>q$ such that
$$
\cW_{\gamma^{\alpha-q},\gamma^{\alpha+q}}(u,X)>\delta\, .
$$
Since the Struwe-energy is bounded by $\Lambda_2$ it follows that
\begin{equation}
\label{e:Nb}
K\leq (2q+1)\delta^{-1}\Lambda_2\, .
\end{equation}
Otherwise, there would be at least $\delta^{-1}\Lambda_2$ disjoint intervalls of the form $(\gamma^{\alpha-q},\gamma^{\alpha+q})$ with $\cW_{\gamma^{\alpha-q},\gamma^{\alpha+q}}(u,X)>\delta$. This is an instance of {\it quantitative differentiation} (see \cite{Cheeger_general_perspective} for a general perspective).

 For each point $X\in P_1$, to keep track of its behavior at different scales, we define a $\{0,1\}$-valued sequence $(T_\alpha(X))_{\alpha\geq 1}$ as follows. By definition, $T_\alpha(X)=1$ if $\alpha\leq q$ or $\cW_{\gamma^{\alpha-q},\gamma^{\alpha+q}}(u,X)>\delta$, and $T_\alpha(X)=0$ if $\alpha>q$ and $\cW_{\gamma^{\alpha-q},\gamma^{\alpha+q}}(u,X)\leq\delta$. Then, for each $\beta$-tuple $(T_\alpha^\beta)_{1\leq \alpha\leq \beta}$, we  put
\begin{equation}
\label{Edef}
E_{T^\beta}(u)=\{X\in P_1\, |\, T_\alpha(X)=T_\alpha^\beta \textrm{ for } 1\leq \alpha\leq \beta \}\, .
\end{equation}
A priori there are $2^\beta$ possible sets $E_{T^\beta}(u)$. However, by the above, $E_{T^\beta}(u)$ is empty whenever $T^\beta$ has more than
$$
Q:=(2q+1)\delta^{-1}\Lambda_2+q
$$
nonzero entries. Thus, we have constructed a decomposition of $P_1$ into at most $\beta^Q$ nonempty sets $E_{T^\beta}(u)$.

\subsection{Cone-splitting}

The goal of this subsection is to prove Corollary \ref{c:hmf_inductive} which says, roughly speaking, that at their good scales points line up in a tubular neighborhood of a well defined {\it almost planar} set. Here, the set of points that we call \emph{$\delta$-good at scales between $Ar$ and $r/A$} ($A>1$) is defined as
\begin{equation}
L_{Ar,r/A,\delta}(u)=\{X\in P_1: \cW_{Ar,A^{-1}r}(u,X)\leq\delta\} \, .
\end{equation}

A key role is played by the following cone-splitting principle and its quantitative version (Lemma \ref{l:hmf_cone_splitting}). Similar ideas played a key role in \cite{CheegerNaber_Ricci,CheegerNaber_HarmonicMinimal,CheegerHaslhoferNaber_MCF}.

\vskip2mm
\noindent
{\bf Cone-splitting principle.}
Assume that $\varphi$ is $j$-selfsimilar at $0$ with respect to $W$ and $0$-selfsimilar at $Y=(y,s)\notin W$. Then we have the following implications:
\begin{itemize}
\item If $W=V^j\times\{0\}$ and
\begin{itemize}
\item if $s=0$, then $y\notin V$ and $\varphi$ is $(j+1)$-selfsimilar at $0$ with respect to ${\rm span}\{y,V^j\}\times\{0\}$.
\item if $s\neq 0$ and $y\in V$, then $\varphi$ is $j$-selfsimilar at $0$ and quasistatic with respect to $V^j\times(-\infty, \max\{s,0\}]$.
\item if $s\neq 0$ and $y\notin V$, then $\varphi$ is $(j+1)$-selfsimilar and quasistatic with respect to ${\rm span}\{y,V^j\}\times(-\infty, \max\{s,0\}]$.
\end{itemize}
\end{itemize}
\begin{itemize}
\item If $W=V^j\times(-\infty,T]$ and
\begin{itemize}
\item if $y\in V$, then $s>T$ and $\varphi$ is $j$-selfsimilar at $0$ and quasistatic with respect to $V^j\times (-\infty,s]$.
\item  $y\notin V$, then $\varphi$ is $(j+1)$-selfsimilar and quasistatic with respect to ${\rm span}\{y,V^j\}\times(-\infty, \max\{s,T\}]$.
\end{itemize}
\end{itemize}
\begin{itemize}
\item If $W=V^{j-2}\times\dR$, then $y\notin V$ and $\varphi$ is $(j+1)$-selfsimilar and static with respect to ${\rm span}\{y,V^{j-2}\}\times\dR$.
\end{itemize}

From an argument by contradiction, we immediately obtain the following quantitative refinement.

\begin{lemma}[Cone-splitting lemma]
\label{l:hmf_cone_splitting}
For all $\varepsilon,\rho>0$, $R,\Lambda_1,\Lambda_2<\infty,m$ and $N$ there exists a constant
$\delta=\delta(\varepsilon,\rho,R,\Lambda_1,\Lambda_2,m,N)>0$ with the following property. If $u \in H^1_{\Lambda_1,\Lambda_2}(P_{4R},N)$ satisfies
\begin{enumerate}
 \item $u$ is $(\delta,R,j)$-selfsimilar at $0$ with respect to $W$.
 \item There exists $Y=(y,s)\in P_1\setminus T_{\rho}(W)$ such that $u$ is
$(\delta,2,0)$-selfsimilar at $Y$,
\end{enumerate}
then we have the following implications:
\begin{itemize}
\item If $W=V^j\times\{0\}$ and
\begin{itemize}
\item if $\abs{s}<\rho^2$, then $d(y,V)\geq\rho$ and $u$ is $(\varepsilon,1,j+1)$-selfsimilar at $0$ with respect to ${\rm span}\{y,V^j\}\times\{0\}$.
\item if $\abs{s}\geq\rho^2$ and $d(y,V)<\rho$, then $u$ is $(\varepsilon,1,j)$-selfsimilar at $0$ with respect to $V^j\times(-\infty, \max\{s,0\}]$.
\item if $\abs{s}\geq\rho^2$ and $d(y,V)\geq\rho$, then $u$ is $(\varepsilon,1,j+1)$-selfsimilar at $0$ with respect to ${\rm span}\{y,V^j\}\times(-\infty, \max\{s,0\}]$.
\end{itemize}
\end{itemize}
\begin{itemize}
\item If $W=V^j\times(-\infty,T]$ and
\begin{itemize}
\item if $d(y,V)<\rho$, then $s\geq T+\rho^2$ and $u$ is $(\varepsilon,1,j)$-selfsimilar at $0$ with respect to $V^j\times (-\infty,s]$.
\item if $d(y,V)\geq\rho$, then $u$ is $(\varepsilon,1,j+1)$-selfsimilar at $0$ with respect to ${\rm span}\{y,V^j\}\times(-\infty, \max\{s,T\}]$.
\end{itemize}
\end{itemize}
\begin{itemize}
\item If $W=V^{j-2}\times\dR$, then $d(y,V)\geq \rho$ and $u$ is $(\varepsilon,1,j+1)$-selfsimilar with respect to ${\rm span}\{y,V^{j-2}\}\times\dR$.
\end{itemize}
\end{lemma}

Using also Lemma \ref{t:hmf_monotone_rigidity}, by induction/contradiction we now obtain:

\begin{corollary}[Line-up in tubular neighborhoods]
\label{c:hmf_inductive}
For all $\mu,\nu>0$, $\Lambda_1,\Lambda_2<\infty,m$ and $N$ there exist
$\delta=\delta(\mu,\nu,\Lambda_1,\Lambda_2,m,N)>0$ and $A=A(\mu,\nu,\Lambda_1,\Lambda_2,m,N)<\infty$ such that the following holds:  If $u \in H^1_{\Lambda_1,\Lambda_2}(P_4,N)$ and $X\in L_{Ar,r/A,\delta}(u)$ for some $r\leq 1/A$, then there exists $0\leq \ell\leq m+2$ and $W_X^\ell$ such that
\begin{enumerate}
\item $u$ is $(\mu,r,\ell)$-selfsimilar at $X$ with respect to $W^\ell_{X}$\, ,
\item $L_{Ar,r/A,\delta}(u)\cap P_{r}(X)\subseteq T_{\nu r}(W^\ell_{X})$\, .
\end{enumerate}
\end{corollary}

To deal with the quasistatic case we also need the following lemma.

\begin{lemma}[Quantitative behavior in the quasistatic case]
\label{l:hmf_quasistatic}
For all $\varepsilon,\gamma>0$, $\Lambda_1,\Lambda_2<\infty,m$ and $N$ there exists $\delta=\delta(\varepsilon,\gamma,\Lambda_1,\Lambda_2,m,N)>0$, such that the following holds: If $u\in H^1_{\Lambda_1,\Lambda_2}(P_4,N)$ is $(\delta,1,\ell)$-selfsimilar at $0$ with respect to $W=V^\ell\times(-\infty,T]$ and if $Y=(y,s)\in P_{1-2\gamma}$ with $s\leq T-(2\gamma)^2$ then $u$ is $(\varepsilon,\gamma,\ell+2)$-selfsimilar at $Y$ with respect to $W=(y+V^\ell)\times\dR$.
\end{lemma}

\begin{proof}
If not, passing to limits we obtain $\varphi$ that is $\ell$-selfsimilar on $P_1$ with respect to $W=V^\ell\times (-\infty,T]$ and a point $Y=(y,s)\in \overline{P}_{1-2\gamma}$ with $s\leq T-(2\gamma)^2$ such that $\varphi$ is not $(\varepsilon,\gamma,\ell+2)$-selfsimilar at $Y$ with respect to $W=(y+V^\ell)\times\dR$, a contradiction.
\end{proof}

\subsection{Conclusion of the argument}\label{hmf_conclusion}
\begin{proof}[Proof of Theorem \ref{t:hmf_quant_strat}]
Let $\varepsilon,\eta,\Lambda_1,\Lambda_2,m,N$ be as in the statement of the theorem. It is convenient to choose $\gamma:=c_0(m)^{-\frac{2}{\varepsilon}}$, where $c_0(m)$ is a geometric constant that only depends on the dimension and will appear below (roughly a doubling constant). Now we apply Corollary \ref{c:hmf_inductive} with $\nu=\gamma/2$ and $\mu\leq\eta$ small enough such that also the below application of Lemma \ref{l:hmf_quasistatic} is justified, and get constants $\delta$ and $A$. Choose an integer $q$, such that $\gamma^q\leq 1/(2A)$. Setting $Q:=\lfloor (2q+1)\delta^{-1}\Lambda_2\rfloor+q$, from the argument in Section \ref{ss:energy_decomposition}, for any $u$ satisfying the assumptions of the theorem we get a decomposition of $P_1$ into at most $\beta^Q$ nonempty sets $E_{T^\beta}(u)$.

\begin{lemma}[Covering lemma]
\label{l:hmf_covering}
There exists $c_0(m)<\infty$ such that each set $\cS^j_{\eta,\gamma^{\beta}}(u)\cap E_{T^\beta}(u)$ can be covered by at most $c_0(c_0\gamma^{-(m+2)})^Q(c_0\gamma^{-j})^{\beta-Q}$ balls of radius $\gamma^\beta$.
\end{lemma}

\begin{proof}
We will recursively define a covering. For $\beta=0$ pick some minimal covering of $\cS^j_{\eta,\gamma^{0}}(u)$ by balls of radius $1$ with centers in  $\cS^j_{\eta,\gamma^{0}}(u)\cap P_1$. Note that $\cS^j_{\eta,\gamma^{\beta+1}}(u)\subset \cS^j_{\eta,\gamma^{\beta}}(u)$. Let $T^{\beta}$ be the $\beta$-tupel obtained from dropping the last entry from $T^{\beta+1}$. Then we also have $E_{T^{\beta+1}}(u)\subset E_{T^{\beta}}(u)$.

\vskip2mm
\noindent
{\bf Recursion step.}
For each ball $P_{\gamma^{\beta}}(X)$ in the covering of $\cS^j_{\eta,\gamma^{\beta}}(u)\cap E_{T^{\beta}}(u)$,
take a minimal covering of $P_{\gamma^{\beta}}(X)\cap \cS^j_{\eta,\gamma^{\beta+1}}(u)\cap  E_{T^{\beta+1}}(u)$
by balls of radius $\gamma^{\beta+1}$ with centers in
$P_{\gamma^{\beta}}(X)\cap \cS^j_{\eta,\gamma^{\beta+1}}(u)\cap  E_{T^{\beta+1}}(u)$.
\vskip2mm

Let us now explain that this covering has indeed the desired properties. First observe that, for all $\beta$, the number of balls in a minimal covering from the recursion step is at most
\begin{equation}
c(m)\gamma^{-(m+2)}\, .
\end{equation}
However, if $T_\beta(X)=0$, then $X\in L_{2A\gamma^{\beta},2\gamma^\beta/A,\delta}(u)$ and Corollary \ref{c:hmf_inductive} gives us $0\leq \ell\leq n+2$ and $W_X^\ell$ such that
\begin{enumerate}
\item $u$ is $(\mu,2\gamma^\beta,\ell)$-selfsimilar at $X$ with respect to $W^\ell_{X}$\, ,
\item $L_{2A\gamma^\beta,2\gamma^\beta/A,\delta}(u)\cap P_{2\gamma^\beta}(X)\subseteq T_{\gamma^{\beta+1}}(W^\ell_{X})$ .
\end{enumerate}
Since $X\in\cS^j_{\eta,\gamma^{\beta}}(u)$ we must have $\ell\leq j$. Since $E_{T^\beta}(u)\subset L_{2A\gamma^\beta,2\gamma^\beta/A,\delta}(u)$ this implies the following better estimate for the number of balls in a minimal covering:
\begin{equation}\label{hmf:bettervolume}
c(m)\gamma^{-j}\, .
\end{equation}
Indeed, the estimate is clear in the cases $W_X^\ell=(x+V^\ell)\times\{t\}$ and $W_X^\ell=(x+V^{\ell-2})\times\dR$, the case $W_X^\ell=(x+V^{\ell})\times(-\infty,T]$ requires some extra thought, but in fact only if $T\leq t+(2\gamma^\beta)^2$ and $\ell\geq j-1$ which we will assume now. So, if $Y=(y,s)\in P_{\gamma^{\beta}}(X)\cap \cS^j_{\eta,\gamma^{\beta+1}}(u)\cap  E_{T^{\beta+1}}(u)$, then by Lemma \ref{l:hmf_quasistatic} we conclude $s\geq T-(2\gamma^{\beta+1})^2$ and thus (\ref{hmf:bettervolume}) holds also in the quasistatic case.
By the quantitative differentiation argument, the better estimate (\ref{hmf:bettervolume}) applies with at most $Q$ exceptions. This proves the lemma.
\end{proof}

We will now conclude the proof of Theorem \ref{t:hmf_quant_strat} by estimating the volume of the covering. The volume of balls in $\dR^{m,1}$ satisfies
\begin{equation}
\label{e:harm_volume}
\Vol(P_{\gamma^\beta}(X))= w_m(\gamma^{\beta})^{m+2}\, ,
\end{equation}
which together with the choice of $\gamma$ and the fact that polynomials grow slower than exponentials, i.e. with
$$
c_0^\beta=(\gamma^\beta)^{-\frac{\varepsilon}{2}}\, ,
$$
$$
\beta^Q\leq c(m,Q)(\gamma^\beta)^{-\frac{\varepsilon}{2}}\, ,
$$
gives (recalling again the decomposition of $P_1$ into at most $\beta^Q$ nonempty sets $\cS^j_{\eta,\gamma^{\beta}}(u)\cap E_{T^\beta}(u)$ and the Covering Lemma \ref{l:hmf_covering})
\begin{equation}
\label{e:main}
\begin{aligned}
\Vol(\cS^j_{\eta,\gamma^\beta}(u))
&\leq \beta^Q \left[c_0  (c_0\gamma^{-(m+2)})^{Q}  (c_0\gamma^{-j})^{\beta-Q} \right] w_m(\gamma^\beta)^{m+2}\\
&\leq c(m,Q,\varepsilon)  \beta^Q c_0^{\beta}(\gamma^\beta)^{m+2-j}\\
&\leq c(m,Q,\varepsilon)(\gamma^\beta)^{m+2-j-\varepsilon}\, .
\end{aligned}
\end{equation}
From the above, for all $0<r<1$, we get
$$
 \begin{aligned}
 \Vol(\cS^j_{\eta,r}(u))
&\leq  c(m,Q,\varepsilon) (\gamma^{-1}r)^{m+2-j-\varepsilon}\\
& \leq c(\varepsilon,\eta,\Lambda_1,\Lambda_2,m,N) r^{m+2-j-\varepsilon}\, .
\end{aligned}
$$
It follows that
$$
\Vol\left(T_r(\cS^j_{\eta,r}(u))\right)\leq Cr^{m+2-j-\varepsilon}\qquad (0<r<1)\, .
$$
for another constant $C=C(\varepsilon,\eta,\Lambda_1,\Lambda_2,m,N)$, and this finishes the proof of Theorem \ref{t:hmf_quant_strat}.
\end{proof}

\section{Higher regularity for certain targets}\label{s:hmf_higher_regularity}

To finish the proofs of Theorem \ref{t:harmonic_map_flow_higher_regularity} and Corollary \ref{c:harmonic_map_flow_higher_regularity}, we prove the following $\varepsilon$-regulartiy lemma. Roughly speaking, it says that enough approximate degrees of symmetry imply regularity.

\begin{lemma}[$\varepsilon$-regularity]
\label{l:hmf_epsilon_regularity}
If $N$ does not admit harmonic two-spheres, then there exists an $\varepsilon=\varepsilon(m,N,\Lambda_1,\Lambda_2,R)>0$ such that the following holds.
Every Chen-Struwe solution $u\in H^1_{\Lambda_1,\Lambda_2}(P_{4R},N)$ that is $(\varepsilon,2r,m-1)$-selfsimilar at $0$ for some $r\leq R$, satisfies $r_u(0)\geq r$.

More generally, if $N$ does not admit harmonic $\ell$-spheres for $\ell=2,\ldots, k$ and quasi-harmonic  $\ell$-spheres for $\ell=3,\ldots, k-1$, then it sufficies to assume that u is  $(\varepsilon,r,m-k+1)$-selfsimilar.
\end{lemma}

\begin{proof}
If not, there are Chen-Struwe solutions $u_\alpha:P_2\rightarrow N$ with bounded scale invariant energy that are $(\alpha^{-1},2,m-1)$-selfsimilar at $0$, but such that $r_{u_\alpha}(0)<1$.
By \cite[Prop. 4.1]{LW1} after passing to a subsequence we can find a strong limit $u_\alpha\to u$ in $H^1_{\textrm{loc}}(P_2,N)$. This limit is a $(m-1)$-selfsimilar weak solution of the harmonic map flow that satisfies the monotonicity formula and the standard $\eps$-regularity lemma, and whose singular set has vanishing $m$-dimensional parabolic Hausdorff measure, see again \cite[Prop. 4.1]{LW1}.
Thus $u$ must be of the form $u(x,t)=\psi(\frac{y}{\abs{y}})$, where we have written $x=(y,z)$ with $y\in \R^3$. Since by assumption the target does not admit harmonic two-spheres, $u$ must be constant. Then the standard $\eps$-regularity lemma \cite[Thm. 5.3]{Struwe_monotonicity} together with the strong convergence in $H^1_{\textrm{loc}}$ implies that $r_{u_\alpha}(0)>1$ for $\alpha$ large enough; a contradiction.
Arguing similarly, the more general statement follows by induction.
\end{proof}

\begin{proof}[Proof of Theorem \ref{t:harmonic_map_flow_higher_regularity}]
Since the initial map is smooth, the regularity scale is bounded below for small times.
Thus, by covering $M\times\dR_+$ with parabolic balls, we can reduce the problem to the local setting of Theorem \ref{t:hmf_quant_strat}.
After this reduction, Lemma \ref{l:hmf_epsilon_regularity} implies that $\cB_r(u)\subseteq \cS^{m-k}_{\eta,r}(u)$ for $\eta$ small enough, and the claim follows from the volume estimate of Theorem \ref{t:hmf_quant_strat}.
\end{proof}

\begin{proof}[Proof of Corollary \ref{c:harmonic_map_flow_higher_regularity}]
Using the layer-cake formula, this follows immediately from Theorem \ref{t:harmonic_map_flow_higher_regularity}, and standard interior estimates (Remark \ref{r:hmf_reg_scale}).
\end{proof}

\section{Examples without harmonic $2$-spheres}\label{s:examples}

We end by considering compact Riemannian manifolds $ N$ which admit no harmonic $2$-spheres.  It has become a common phenomena in the theory of harmonic maps that when the target space $N$ does not admit any harmonic $2$-spheres then the regularity of a harmonic mapping improves.  In particular, a {\it stationary} harmonic map has the same regularity properties as a minimizing harmonic map, see \cite{Lin_elliptic} and Theorem \ref{t:harmonic_map_flow_higher_regularity} and Corollary \ref{t:stationary_regularity}.  It is known from the work of Sacks-Uhlenbeck that in this case the universal cover of $N$ must be contractible. However in the literature at this point the only examples of such spaces $N$ have nonpositive sectional curvature.  In particular, for these examples other methods can 
be used to arrive at much stronger regularity results.  The goal of this Section is simply to record a few examples of compact manifolds without harmonic $2$-spheres which do not have nonpositive sectional curvature.

\begin{theorem}
Let $N$ be one of the following:
\begin{enumerate}
\item A $3$-dimensional infranil manifold equipped with a left invariant metric.

\item A  $3$-dimensional solv-manifold with a left invariant metric.
\end{enumerate}
Then $N$ admits no harmonic $2$-spheres.  In particular, products of such spaces with each other or with spaces of nonpositive sectional curvature also admit no harmonic $2$-spheres.
\end{theorem}

\begin{question}
Do there exist any nilmanifolds $N$ with left invariant metrics that admit harmonic $2$-spheres?  What can be said about the existence of harmonic $k$-spheres?
\end{question}

We are indebted to Harold Rosenberg for the following simple argument.

\begin{proof}
View $N$ is the quotient of its universal cover $\tilde N$ by a discrete cocompact group
$\Gamma$ acting on the left by isometries. 
Since $S^2$ is simply connected we can lift any mapping into $N$ to $\tilde N$,
Since the image of a harmonic $2$-sphere is a branched minimal surface, it suffices to show 
that $\tilde N$ contains no branched minimal $2$-spheres.
Let us first deal with the case when $N$ is nilpotent.  In particular, unless $\tilde N$ is euclidean space, it
 is the Heisenberg group.  If $\eta$ is the lie algebra of $\tilde N$ equipped with an inner product, let $e_1,e_2,e_3$ be an orthonormal basis with $[\eta,\eta]=\text{span}\{e_3\}$.  A standard computation shows that the closed normal subgroup $\tilde N'\subseteq \tilde N$ associated to the subalgebra $\eta'\equiv\text{span}\{e_2,e_3\}$ is minimal.  The family of hypersurfaces $n\cdot \tilde N'$ with $n\in N$ gives a foliation of $N$ by minimal hypersurfaces which is preserved by the left action of $\Gamma$, with 
orbit space $\tilde N/\tilde N'=\bf R$. If $S^2\to \tilde N$ is 
a branched minimal sphere, then there exists $n\in N$ such that $n\cdot N'$ touches the $S^2$ from one side.  By a standard maximum principle argument, this is a contradiction unless the image of $S^2$ is a point.

If $N$ is a compact solv-manifold, then we can pick an orthonormal basis of the lie algebra $e_1,e_2,e_3$ such that $[e_1,e_2]=0$, $[e_1,e_3]=e_1$, $[e_2,e_3]=-e_2$.  If we let $\eta'=\text{span}\{e_1,e_2\}$ be a subalgebra then the brackets allow us to compute the mean curvature of the corresponding closed
normal subgroup as $\langle \nabla_{e_1}e_1,e_3\rangle+\nabla_{e_2}e_2,e_3\rangle = 1-1=0$.  Now one can argue as in the nilpotent case to finish the proof.
\end{proof}

\appendix

\bibliography{CHN_hmf}

\begin{thebibliography}{CN13b}

\bibitem[Che12]{Cheeger_general_perspective}
J.~Cheeger.
\newblock Quantitative differentiation: a general formulation.
\newblock {\em Comm. Pure Appl. Math.}, 65(12):1641--1670, 2012.

\bibitem[CHN]{CheegerHaslhoferNaber_MCF}
J.~Cheeger, R.~Haslhofer, and A.~Naber.
\newblock Quantitative stratification and the regularity of mean curvature
  flow.
\newblock {\em Geom. Funct. Anal. (to appear)}.

\bibitem[CLL95]{ChenLiLin_suitable}
Y.~M. Chen, J.~Li, and F.-H. Lin.
\newblock Partial regularity for weak heat flows into spheres.
\newblock {\em Comm. Pure Appl. Math.}, 48(4):429--448, 1995.

\bibitem[CN13a]{CheegerNaber_Ricci}
J.~Cheeger and A.~Naber.
\newblock Lower bounds on {R}icci curvature and quantitative behavior of
  singular sets.
\newblock {\em Invent. Math.}, 191(2):321--339, 2013.

\bibitem[CN13b]{CheegerNaber_HarmonicMinimal}
J.~Cheeger and A.~Naber.
\newblock Quantitative stratification and the regularity of harmonic maps and
  minimal currents.
\newblock {\em Comm. Pure and Appl. Math.}, 66(6):965--990, 2013.

\bibitem[CS89]{ChenStruwe_existence}
Y.~M. Chen and M.~Struwe.
\newblock Existence and partial regularity results for the heat flow for
  harmonic maps.
\newblock {\em Math. Z.}, 201(1):83--103, 1989.

\bibitem[ES64]{Eells_Sampson}
J.~Eells, Jr. and J.~H. Sampson.
\newblock Harmonic mappings of {R}iemannian manifolds.
\newblock {\em Amer. J. Math.}, 86:109--160, 1964.

\bibitem[Fel94]{Feldman}
M.~Feldman.
\newblock Partial regularity for harmonic maps of evolution into spheres.
\newblock {\em Comm. Partial Differential Equations}, 19(5-6):761--790, 1994.

\bibitem[Lin99]{Lin_elliptic}
F.-H. Lin.
\newblock Gradient estimates and blow-up analysis for stationary harmonic maps.
\newblock {\em Ann. of Math. (2)}, 149(3):785--829, 1999.

\bibitem[LW99]{LW1}
F.-H. Lin and C.~Y. Wang.
\newblock Harmonic and quasi-harmonic spheres.
\newblock {\em Comm. Anal. Geom.}, 7(2):397--429, 1999.

\bibitem[Str88]{Struwe_monotonicity}
M.~Struwe.
\newblock On the evolution of harmonic maps in higher dimensions.
\newblock {\em J. Differential Geom.}, 28(3):485--502, 1988.

\bibitem[SU81]{SacksUhlenbeck}
J.~Sacks and K.~Uhlenbeck.
\newblock The existence of minimal immersions of {$2$}-spheres.
\newblock {\em Ann. of Math. (2)}, 113(1):1--24, 1981.

\end{thebibliography}

\bibliographystyle{alpha}

\vspace{10mm}
{\sc Jeff Cheeger and Robert Haslhofer, Courant Institute of Mathematical Sciences, New York University, 251 Mercer Street, New York, NY 10012, USA}\\

{\sc Aaron Naber, Department of Mathematics, Massachusetts Institute of Technology, 77 Massachusetts Avenue, Cambridge, MA 02139, USA}\\

\emph{E-mail:} robert.haslhofer@cims.nyu.edu, cheeger@cims.nyu.edu, anaber@math.mit.edu

\end{document}